\documentclass{article}

\usepackage{microtype}
\usepackage{graphicx}
\usepackage{subfigure}
\usepackage{booktabs} 
\usepackage{mathtools}
\usepackage{amsmath}
\usepackage{amsmath,amssymb,amsthm,mathrsfs,amsfonts,dsfont, tikz}
\usepackage{makecell}
\usepackage{hyperref}

\usepackage[accepted]{icml2019}

    \usepackage{booktabs}

    \usepackage{color}

    \newtheorem{theorem}{Theorem}
    
    \newtheorem{lemma}{Lemma}

     \newtheorem{assumption}{Assumption}

    \def \R {\mathbb{R}}

    \def \x {\mathbf{x}}
    
    \def \E {\mathrm{E}}
    
    \def \x {\mathbf{x}}
    
    \def \a {\mathbf{a}}

    \def \z {\mathbf{z}}

    \def \y {\mathbf{y}}

    \def \I {\mathbb{I}}

    \def \xh {\widehat{\x}}

    \def \Lh {\hat{L}}
    \def \Dt {\widetilde{D}}

\icmltitlerunning{Katalyst for  Non-Convex Problems}

\begin{document}

\twocolumn[
\icmltitle{Katalyst: Boosting Convex Katayusha\\ for  Non-Convex Problems with a  Large Condition Number}





\icmlsetsymbol{equal}{*}
\begin{icmlauthorlist}
\icmlauthor{Zaiyi Chen}{cainiao,ustc}
\icmlauthor{Yi Xu}{uiowa}
\icmlauthor{Haoyuan Hu}{cainiao}
\icmlauthor{Tianbao Yang}{uiowa}
\end{icmlauthorlist}

\icmlaffiliation{cainiao}{Cainiao AI, China}
\icmlaffiliation{uiowa}{University of Iowa, USA}
\icmlaffiliation{ustc}{University of Science and Technology of China, China}
\icmlcorrespondingauthor{}{zaiyi.czy@alibaba-inc.com}
\icmlkeywords{Non-convex Optimization, Variance Reduction}

\vskip 0.3in
]

\printAffiliationsAndNotice

\begin{abstract}
An important class of non-convex objectives that has wide applications in machine learning consists of  a sum of $n$ smooth functions and a non-smooth convex function. Tremendous studies have been devoted to conquering these problems by leveraging one of the two types of variance reduction techniques, i.e., SVRG-type that computes a full gradient occasionally and SAGA-type that maintains $n$ stochastic gradients at every iteration.  In practice, SVRG-type is preferred to SAGA-type due to its potentially less memory costs. 
An interesting question that has been largely ignored is how to improve the complexity of variance reduction methods for problems with a large condition number that measures the degree to which the objective is close to a convex function. 
In this paper, we present a simple but non-trivial boosting of a  state-of-the-art SVRG-type method for convex problems (namely Katyusha) to enjoy an improved complexity for solving non-convex problems with a large condition number (that is close to a convex function). To the best of our knowledge, its complexity has the best dependence on $n$ and the degree of non-convexity, and also matches that of a recent SAGA-type accelerated stochastic algorithm for a constrained non-convex smooth optimization problem.  Numerical experiments verify the effectiveness of the proposed algorithm in comparison with its competitors.  

\end{abstract}

 \section{Introduction}
The problem of interest in this paper belongs to the following class of non-convex optimization problems: 
\begin{align}\label{eqn:P}
\min_{\x\in\R^d} \phi(\x):= \frac{1}{n}\sum_{i=1}^nf_i(\x) + \psi(\x),
\end{align}
where each $f_i$ is a $L$-smooth  function, and $\psi(\x)$ is a ``simple" closed convex function whose proximal mapping can be efficiently computed.  The above problem covers constrained and non-constrained smooth optimization as special cases when $\psi(\x)$ is the indicator function of a convex set and $\psi=0$. 
This problem has broad applications in machine learning, and has been  studied by numerous  papers~\citep{Reddi:2016:SVR:3045390.3045425,DBLP:conf/cdc/ReddiSPS16,reddi2016proximal,DBLP:journals/corr/abs/1805.05411,natasha1a,DBLP:conf/icml/ZhuH16}.  A  number of stochastic algorithms were proposed by utilizing the finite-sum structure of the problem and smoothness of $f_i$ to derive faster convergence than stochastic gradient methods. These algorithms are  based on two well-known variance-reduction techniques, namely the SVRG-type variance reduction~\citep{NIPS2013_4937} and the SAGA-type variance reduction~\citep{DBLP:conf/nips/DefazioBL14,DBLP:conf/nips/RouxSB12}. The key difference between these two variance reduction techniques is that SVRG uses a full gradient that is computed periodically and SAGA uses a full gradient that is computed from its maintained historical gradients for each component $f_i$. Due to this difference, SAGA might require much higher memory than SVRG for many problems, which renders algorithms of SVRG-type more favorable than  algorithms of SAGA-type. 

\begin{table*}[t]
		\caption{Comparison of gradient complexities of variance reduction based algorithms for finding $\epsilon$-stationary point of~(\ref{eqn:P}) with $\psi=0$. The best complexity result for each setting is marked in red color.  The top two algorithms, namely SAGA and RapGrad use the SAGA-type variance reduction technique, while others use the SVRG-type variance reduction technique. $\widetilde O(\cdot)$ hides some logarithmic factor. $^*$ marks the result is only valid when $L/\mu\leq\sqrt{n}$. 
		 }
		\centering
		\label{tab:2}
		\scalebox{1}{\begin{tabular}{l||l|ll}
			\toprule
			Algorithms & $L/\mu\geq \Omega(n)$ &$L/\mu\leq O(n)$&Non-smooth $\psi$\\
			\midrule
			SAGA~\citep{reddi2016proximal}&$O(n^{2/3}L/\epsilon^2)$&$O(n^{2/3}L/\epsilon^2)$&Yes\\ 
			RapGrad~\citep{DBLP:journals/corr/abs/1805.05411} &${\color{red}\widetilde O(\sqrt{nL\mu}/\epsilon^2)}$&$\widetilde O((\mu n + \sqrt{nL\mu})/\epsilon^2)$&indicator function\\ 
			\midrule
			SVRG~\citep{reddi2016proximal} &$O(n^{2/3}L/\epsilon^2)$&$O(n^{2/3}L/\epsilon^2)$&Yes\\ 
			Natasha1~\citep{DBLP:conf/icml/Allen-Zhu17}&NA&${O( n^{2/3}L^{1/3}\mu^{2/3}/\epsilon^2)}^*$&Yes\\ 
            RepeatSVRG~\citep{DBLP:conf/icml/Allen-Zhu17}&$\widetilde O(n^{3/4}\sqrt{L\mu}/\epsilon^2)$&$\widetilde O((\mu n + n^{3/4}\sqrt{L\mu})/\epsilon^2)$&Yes\\
                    4WD-Catalyst~\citep{pmlr-v84-paquette18a}&$O(n L/\epsilon^2)$&$O(nL/\epsilon^2)$&Yes\\
                   SPIDER~\citep{DBLP:conf/nips/FangLLZ18}&$O(\sqrt{n}L/\epsilon^2)$&$O(\sqrt{n}L/\epsilon^2)$&No\\
                   SNVRG~\citep{DBLP:conf/nips/ZhouXG18}&$O(\sqrt{n}L/\epsilon^2)$&$O(\sqrt{n}L/\epsilon^2)$&No\\
           Katalyst  (this work)&${\color{red}\widetilde O(\sqrt{n L\mu}/\epsilon^2)}$&$ \widetilde O((\mu n + L)/\epsilon^2)$&Yes\\
			\bottomrule
		\end{tabular}}
			\vspace*{-0.2in}
	\end{table*}
Since the proposal of non-convex SVRG for solving non-convex problems in the form of~(\ref{eqn:P}) or its special case with $\psi=0$~\citep{Reddi:2016:SVR:3045390.3045425,DBLP:conf/icml/ZhuH16}, several studies have tried to improve its complexity in terms of the number of components $n$~\citep{DBLP:conf/nips/FangLLZ18,DBLP:conf/nips/ZhouXG18}.  To the best of our knowledge, the state-of-the-art gradient complexity~\footnote{the number of stochastic gradient computations} of SVRG-type methods for finding a solution $\x$ such that $\E[\|\nabla \phi(\x)\|]\leq \epsilon$ under the condition  $\psi=0$ and $\epsilon\leq 1/\sqrt{n}$ is given by $O(L\sqrt{n}/\epsilon^2)$. It was also shown in \citep{DBLP:conf/nips/FangLLZ18} that such a complexity is a lower bound for the problem~(\ref{eqn:P}), hence it cannot be improved in general. 

However, most of previous studies have ignored the degree of non-convexity of each component function with few exceptions discussed later. Intuitively, a non-convex function that is closer to a convex function should be easily optimized.  A natural way to measure the degree of non-convexity is by considering a notion of $\mu$-weak convexity.   In particular a function $f$ is said to be $\mu$-weakly convex if $f(\x) + \frac{\mu}{2}\|\x\|^2$ is a convex function for $\mu>0$, where $\|\cdot\|$ denotes the Euclidean norm.  If $f$ is twice-differentiable, $\mu$-weak convexity is equivalent to  that $\nabla^2 f(\x)\geq - \mu I$. Hence, the smaller the $\mu$,  the closer the function to a convex function.  For a smooth function with $L$-Lipchitz continuous gradient we define  the condition number as $L/\mu$. Therefore, an interesting question is whether the gradient complexity of a SVRG-type method can be further improved for~(\ref{eqn:P}) with $\mu$-weakly convex functions $f_i$ when $\mu$  is very small. In another word, whether the gradient complexity can be made dependent on $\mu$ such that the closer $f_i$ is  to a convex function the smaller is the complexity.  In this paper, we provide an affirmative answer to this question. We show that when the condition number of each $f_i$ is large (i.e, $L/\mu\geq \Omega(n)$), we can improve the complexity to $\widetilde O(\sqrt{nL\mu}/\epsilon^2)$, which is better than that reported in~\cite{DBLP:conf/nips/FangLLZ18,DBLP:conf/nips/ZhouXG18}. To the best of our knowledge, this is the best result for a SVRG-type method for solving problem~(\ref{eqn:P}) under a large condition number, which also matches that of a recent work focusing on developing an accelerated  SAGA-type method for solving constrained non-convex smooth optimization~\cite{DBLP:journals/corr/abs/1805.05411}.  We also establish a gradient complexity of $\widetilde O(\mu n/\epsilon^2)$ in the case of $L/\mu<O(n)$, which improves the complexity of \cite{DBLP:conf/nips/FangLLZ18,DBLP:conf/nips/ZhouXG18} when $\mu/L\leq 1/\sqrt{n}$, and is also  slightly better than that of~\cite{DBLP:journals/corr/abs/1805.05411}. The proposed algorithm is a simple but non-trivial boosting of convex Katyusha~\cite{DBLP:conf/stoc/Zhu17}. The idea is by calling convex Katyusha for solving a sequence of regularized convex problems, which is  similar to that used in the Catalyst technique for speeding up convex optimization~\cite{lin2015universal}. However, the key difference and novelty of the proposed algorithm is that we do not use any extrapolation step and the acceleration is simply achieved by carefully choosing the parameters (i.e., the number of epochs and the number of iterations for the inner loop) for convex Katyusha that are adaptive to the $\mu$-weak convexity of the problem.  We refer to the proposed algorithm as Katalyst.

Before ending this section, we present a motivating example of the considered easy non-convex problems with a large condition number.  Let us consider least-squares regression with non-convex sparsity-promoting regularizers: 
\begin{align}\label{eqn:ncsr}
\min_{\x\in\R^d}\frac{1}{n}\sum_{i=1}^n\ell(\a_i^{\top}\x, b_i)+ \lambda R(\x),
\end{align}
where $(\a_i, b_i), i=1,\ldots, n$ denote a set of $n$ observed data with $\a_i\in\R^d$ representing the feature vector and $b_i\in\R$ representing the label of the $i$-th example, $\ell(\a_i^{\top}\x, b_i)=(\a_i^{\top}\x - b_i)^2$, $R(\x)$ denotes a non-convex regularizer that enforces sparsity and $\lambda>0$ is a regularization parameter. Commonly used non-convex sparsity-promoting regularizers include logarithmic sum penalty~\cite{Candades2008}, transformed $\ell_1$ norm~\cite{DBLP:journals/corr/ZhangX14},  smoothly clipped absolute deviation (SCAD) regularization~\cite{CIS-172933}, minimax concave penalty (MCP) regularization~\cite{cunzhang10}. All of these regularizers can be written as a (scaled) $\ell_1$ norm  minus a differentiable smooth function. Let us consider the logarithmic sum penalty  $R(\x) = \sum_{i=1}^d\log(|\x_i| + \theta)$. It can be written as $R(\x) = 1/\theta\|\x\|_1 + R_2(\x)$, where $R_2(\x) = \sum_{i=1}^d (\log(|x_i|  + \theta) - |x_i|/\theta )$. It was shown that $R_2$ is a differentiable smooth non-convex function with a smoothness parameter $\mu = \frac{1}{\theta^2}$~\cite{Wen2018}. In order to formulate the problem as~(\ref{eqn:P}), we can defined $f_i(\x) = (\a_i^{\top}\x - b_i)^2/2  + \lambda R_2(\x)$ and $\psi(\x) = \lambda\|\x\|_1/\theta$. Thus, we have $f_i$ is $\mu = \lambda/\theta^2$-weakly convex and $L = \max_i\|\a_i\|^2 + \lambda/\theta^2$-smooth. When the regularization parameter $\lambda$ is very small,  then the condition number is very large. Similar discussions have been applied to other regularizers.

\section{Related Work}\label{sec: related}
Since the proposal of variance reduction techniques were proposed by~\citet{NIPS2013_4937,DBLP:conf/nips/RouxSB12,NIPS2013_4940}, they have received tremendous attention. In this paper  we are mostly interested in non-convex problems. Hence,  below we review some related works for non-convex optimization  in the form of~(\ref{eqn:P}). 

A SVRG-type method for solving non-convex smooth optimization - a special case of~(\ref{eqn:P}) with $\psi=0$ were first proposed  by two research groups independently~\citep{Reddi:2016:SVR:3045390.3045425,DBLP:conf/icml/ZhuH16}. The gradient complexity of non-convex SVRG is given by $O(n^{2/3}L/\epsilon^2)$ for finding an $\epsilon$-stationary solution such that $\E[\|\nabla\phi(\x)\|]\leq \epsilon$. It was later generalized to solving the general case~(\ref{eqn:P}) with $\psi$ being a non-smooth convex function by~\citep{DBLP:conf/cdc/ReddiSPS16,reddi2016proximal}, which also includes a SAGA-type method.  There are two basic variants of SVRG proposed in~\cite{Reddi:2016:SVR:3045390.3045425,DBLP:conf/cdc/ReddiSPS16,reddi2016proximal} one with a large  mini-batch size ($n^{2/3}$) and one with a small step size $\Theta(1/n^{2/3}L)$. In the first variant, the step size can be set to a large value $\Theta(1/L)$. In the second variant, the mini-batch size can be set to $1$. However, neither variant is practical, especially with a small step size $\Theta(1/n^{2/3}L)$, which usually leads to slow convergence in practice. In contrast, the proposed method uses a large step size $\Theta(1/L)$ and allows for using a mini-batch size of $1$. 

Recently, there are several improvements on the gradient complexity for SVRG-type methods in terms of dependence on $n$. In particular, two new SVRG-type algorithms were proposed in~\citep{DBLP:conf/nips/FangLLZ18,DBLP:conf/nips/ZhouXG18}, namely SPIDER and stochastic nested variance reduction for solving the problem~(\ref{eqn:P}) with $\psi=0$. The gradient complexity of both algorithms is given by $O(\sqrt{n}L/\epsilon^2)$ for finding an $\epsilon$-stationary solution when $\epsilon\leq O(1/\sqrt{n})$.

Few works have taken the $\mu$-weak convexity of individual functions $f_i$ into account for the development of variance reduction methods~\cite{DBLP:conf/icml/Allen-Zhu17,DBLP:journals/corr/abs/1805.05411}. Under the weakly convex assumption, \citet{DBLP:conf/icml/Allen-Zhu17} proposed a novel acceleration of SVRG-style method, namely Natasha1, which established a state-of-the-art gradient complexity when condition number is small, i.e. $\sqrt{n}\geq L/\mu$.
In the same paper, \citet{DBLP:conf/icml/Allen-Zhu17} also discussed another method, namely RepeatSVRG~\footnote{After the preliminary version of this manuscript was finished, it was brought to our attention that the updated arXiv manuscript \citep[V5]{natasha1a} reported a new result for RepeatSVRG for our considered problem different from its proceedings version, which is in the same order as the result achieved in this work. It is less practical than our method.}, which could converge faster than Natasha1 under a large condition number.  The proposed method is more practical than RepeatSVRG in that it does not require setting $\epsilon$ aprior as in RepeatSVRG. In a more recent work, \citet{DBLP:journals/corr/abs/1805.05411} proposed an SAGA-type method, which has the same gradient complexity of this work except for a worse memory cost. It is the first-work for deriving an $\mu$-dependent complexity of a variance-reduction method for solving smooth non-convex optimization problems. Our work is complementary by developing a SVRG-type method with the same complexity and for solving a broader family of problems with a non-smooth convex function $\psi$. 


It is notable that accelerating the convergence for strongly convex and smooth optimization problems with a large condition number has received a lot of attention in the community~\cite{lin2015universal,DBLP:conf/icml/FrostigGKS15,DBLP:journals/mp/LanZ18,DBLP:conf/stoc/Zhu17}. Recently, \citet{pmlr-v84-paquette18a} also considered extending the Catalyst technique for speeding up convex optimization algorithms to solving non-convex problem~(\ref{eqn:P}). However, their gradient complexity for using SVRG is only $O(nL/\epsilon^2)$, which is worse than our result. Finally, we present a comparison between this work and previous works for solving~(\ref{eqn:P}) in Table~\ref{tab:2}.

\section{Katalyst}
In this section we present the proposed Katalyst  algorithm and its analysis. We first present some notations.  
For simplicity of presentation, we let $f = \sum_{i=1}^nf_i(\x)/n$, and let \[
\text{prox}_{\lambda \psi}(\x) =\arg\min_{\z} \psi(\z) + \frac{1}{2\lambda}\|\z - \x\|^2
\] denote the proximal mapping of a function $\psi$. 
For problem~(\ref{eqn:P}), a point $\x\in\text{dom}(\psi)$ is a first-order stationary point if $ 0\in \partial \phi(\x)$, 
where $\partial\phi$ denotes the partial gradient of $\phi$. However, it is hard for an iterative algorithm to find an exact stationary point with a finite number of iterations. Therefore, some notion of $\epsilon$-stationary is usually considered. 

 In the literature, several notions of $\epsilon$-stationarity were considered by accommodating the non-smooth term $\psi$ in different way.  The first measure is simply using the sub-differentiable of the objective function $\phi$.  Under this measure,  a point $\x$ is said to be $\epsilon$-stationary  if  $\text{dist}(0, \partial \phi (\x))\leq \epsilon$, 
where $\text{dist}$ denotes the Euclidean distance from a point to a set and $\partial \phi(\x) = \nabla f(\x) + \partial \psi (\x)$.  The second measure is using the proximal gradient defined as: 
\begin{align}\label{eqn:pg}
\mathcal G_\eta(\x) = \frac{1}{\eta}(\x - \text{prox}_{\eta \psi}(\x - \nabla f(\x))).
\end{align}
Under this measure,  a point $\x$ is said to be $\epsilon$-stationary  if $\|\mathcal G_\eta(\x)\|^2\leq \epsilon$. This convergence measure has been used in~\cite{reddi2016proximal,DBLP:conf/icml/Allen-Zhu17}. The third stationarity measure that is more general is defined by using a notion of nearly stationary. In particular, a point $\x$ is called $(\epsilon, \delta)$-nearly stationary if there exists a point $\xh$ such that  
\begin{align}
\|\x - \xh\|\leq \delta, \quad \text{dist}(0, \partial \phi(\xh))\leq \epsilon.
\end{align} This convergence measure has been used in~\cite{davis2017proximally,modelweakly18,DBLP:journals/corr/abs/1805.05411,chen18stagewise}. The third convergence measure is more general that covers the first two measures as special cases. This can be easily seen for the first convergence measure with $\xh=\x$ and $\delta=0$. For the second convergence measure, we can show that when  $\|\mathcal G_\eta(\x)\|\leq \epsilon$ holds with $\eta = 1/L$, we have $\|\x - \z\|\leq \epsilon/L$ and $\text{dist}(0, \partial \phi(\z))\leq \|\mathcal G_{\eta}(\x)\| + L\|\x - \z\|\leq 2\epsilon$,  where $\z=\text{prox}_{\eta \psi}(\x - \nabla f(\x))$. 

In this paper, we use the third stationarity measure that is same as that used in~\cite{davis2017proximally,modelweakly18,DBLP:journals/corr/abs/1805.05411,chen18stagewise}, which is more suitable for our algorithm than other measures.  To this end, we introduce the Moreau envelope of $\phi$
\begin{align*}
\phi_\lambda(\x)= \min_{\z} \phi(\z) + \frac{1}{2\lambda}\|\z - \x\|^2.
\end{align*}
Further, the optimal solution to the above problem is  $\text{prox}_{\lambda\phi}(\x)$.  It is known that if $\phi(\x)$ is $\rho$-weakly convex and $\lambda<\rho^{-1}$, then its Moreau envelope $\phi_{\lambda}(\x)$ is $C^1$-smooth with the gradient given by $\nabla \phi_{\lambda}(\x) = \lambda^{-1}( \x  - \text{prox}_{\lambda \phi}(\x))$ (see e.g.~\cite{sgdweakly18}). 
A small norm of $\nabla \phi_{\lambda}(\x)$ has an interpretation that $\x$ is close to $\xh =  \text{prox}_{\lambda\phi}(\x)$ that is $\epsilon$-stationary. In particular for any $\x\in\R^d$, let $\xh = \text{prox}_{\lambda \phi}(\x)$, then we have
\begin{align}\label{eqn:keye}
\left\{
\begin{aligned}
&\phi(\xh)  \leq \phi(\x),\\
&\|\x - \xh\| = \lambda \|\nabla \phi_{\lambda}(\x)\|,
\\
&\text{dist}(0, \partial \phi(\xh))\leq \|\nabla \phi_\lambda(\x)\|.
\end{aligned}
\right.
\end{align}
This means that  a point $\x$ satisfying $\|\nabla \phi_\lambda(\x)\|\leq \epsilon$  is close to a point in distance of $O(\epsilon)$ that is $\epsilon$-stationary. Below, we will prove the convergence in terms of $\|\nabla\phi_\lambda(\x)\|$ for some $\lambda>0$  and $\|\x - \text{prox}_{\lambda \phi}(\x)\|$ as well, which is consistent with that in~\citep{DBLP:journals/corr/abs/1805.05411}. 

\begin{algorithm}[t]
    \caption{Katalyst for Non-Convex Optimization}\label{alg:meta}
    \begin{algorithmic}[1]
        \STATE \textbf{Initialize:} non-decreasing positive weights $\{w_s\}$,  $\x_0\in \text{dom}(\psi)$, $\gamma=(2\mu)^{-1}$
        \FOR {$s = 1,\ldots, S+1$}
        \STATE Let $f_{s}(\cdot) = \phi(\cdot)+\frac{1}{2\gamma}\| \cdot-\x_{s-1}\|^2$
        \STATE $\x_{s} = \text{Katyusha}(f_{s}, \x_{s-1}, K_s, \mu,L+ \mu)$
        \ENDFOR
        \STATE \textbf{Return:} $\x_{\tau+1}$, $\tau$ is randomly chosen from $\{0, \ldots, S\}$ according to probabilities $p_\tau = \frac{w_{\tau+1}}{\sum_{k=0}^{S}w_{k+1}}, \tau=0, \ldots, S$.
    \end{algorithmic}
\end{algorithm}

\begin{algorithm}[t]
    \caption{\text{Katyusha}($f,x_0, K, \sigma,\Lh$)}\label{alg:kat}
\begin{algorithmic}[1]
\STATE \textbf{Initialize: }$\tau_2 = \frac{1}{2},\; \tau_1 = \min\{\sqrt{\frac{n\sigma}{3\hat{L}}},\frac{1}{2}\},\; \eta = \frac{1}{3\tau_1\hat{L}},\;\theta = 1+\eta\sigma,\;m=\lceil\frac{\log(2\tau_1+2/\theta - 1)}{\log\theta}\rceil + 1$
\STATE $y_0=\zeta_0=\widetilde{x}^0\leftarrow x_0$

\FOR{$k=0,\ldots,K-1$}
\STATE $u^k = \nabla \hat{f}(\widetilde{x}^k)$
\FOR{$t=0,\ldots,m-1$}
\STATE $j=km+t$
\STATE $x_{j+1} = \tau_1\zeta_j+\tau_2\widetilde{x}^k+(1-\tau_1-\tau_2)y_j$
\STATE $\widetilde{\nabla}_{j+1} = u^k+\nabla \hat{f}_i(x_{j+1}) - \nabla \hat{f}_i (\widetilde{x}^k)$
\STATE $\zeta_{j+1} = \arg\min_{\zeta} \frac{1}{2\eta}\|\zeta-\zeta_j\|^2+\langle\widetilde{\nabla}_{j+1},\zeta\rangle + \psi(\zeta)$
\STATE $y_{j+1} = \arg\min_{y} \frac{3\Lh}{2}\|y-x_{j+1}\|^2 +\langle \widetilde{\nabla}_{j+1},y\rangle $ \label{alg:kat y}
\ENDFOR
\STATE 
compute $\widetilde{x}^{k+1} = \frac{\sum_{t=0}^{m-1}\theta^t y_{sm+t+1}}{\sum_{j=0}^{m-1}\theta^t}$
\ENDFOR
\STATE \textbf{Output} $\widetilde{x}^K$
\end{algorithmic}
\end{algorithm}
\subsection{Algorithm}
The Katalyst algorithm is presented in Algorithm~\ref{alg:meta}, which falls into the same framework presented in~\citep{chen18stagewise}. The idea is to construct a strongly convex function $f_s$ at each stage and then call a stochastic algorithm (Katyusha here) for approximately solving the constructed function. One may consider directly applying their Theorem 1 to prove the convergence. However, their analysis only concerns the convergence of $\|\nabla\phi_\gamma(\x_\tau)\|$ without explicit considering the convergence of $\|\x - \text{prox}_{\gamma\phi}(\x)\|$, which is important for proving the convergence of $\|\nabla \phi(\x)\|$ when $\psi=0$. By using the second inequality in~(\ref{eqn:keye}), one can bound $2\mu \|\x - \text{prox}_{\gamma\phi}(\x)\|$ by $\|\nabla\phi_\gamma(\x_\tau)\|$. Nevertheless, in the case of $\mu\ll 1$, such analysis will yield much worse gradient complexity than that is achieved below. Hence, we need a more refined analysis of the proposed algorithm with a careful setting of Katyusha for solving each subproblem. 

A modified Katyusha is employed at each stage for solving the regularized subproblem $f_s(\x)$, which is assumed to be $\sigma$-strongly convex and have $\hat L$-Lipschitz continuous gradients for the smooth components. The modified Katyusha is presented in Algorithm~\ref{alg:kat}.  Given the way that $f_s$ is constructed, we can write it as 
\begin{align*}
f_s(\x) = &\frac{1}{n}\sum_{i=1}^n(\underbrace{f_i(\x)  + \frac{\mu}{2}\|\x - \x_{s-1}\|^2}\limits_{\hat f_i(\x)})\\
& + \underbrace{\frac{\gamma^{-1} - \mu}{2}\|\x - \x_{s-1}\|^2 + \psi(\x)}\limits_{\hat \psi(\x)}.
\end{align*}
It is easy to see that $\hat f_i(\x)$ is convex and $\hat L = (L + \mu)$-smooth, and $\hat \psi(\x)$ is $\sigma = (\gamma^{-1} - \mu)$-strongly convex, which satisfy the conditions made in~\citep{DBLP:conf/stoc/Zhu17}. In each call of the modified Katyusha, $\hat f_i$ is considered as the smooth component, and $\hat\psi$ is considered as the non-smooth regularizer.  The key difference between our modified Katyusha and the original Katyusha algorithm for solving smooth and strongly convex problems in~\citep{DBLP:conf/stoc/Zhu17} lies at the setting of $\tau_1$, $m$ and $K$. For example in~\citep{DBLP:conf/stoc/Zhu17}, the value of $\tau_1$ is set to $\tau_1 = \min(\sqrt{m\sigma/3\hat L}, 1/2)$. However, in our modified Katyusha the value of $\tau_1$ is independent of $m$. The value of $m$ is also different from that suggested in~\citep{DBLP:conf/stoc/Zhu17}, which is suggested to $2n$. The value of $K$ (the number of epochs) in the original Katyusha is chosen such that the objective gap is less than $\epsilon$. In our modified Katyusha, it is set to make sure that the objective function $f_s(\x)$ is decreased by a sufficient amount. Actually, we do not solve $\min_{\x}f_s(\x)$ to an $\epsilon$-accuracy level in terms of the objective value.  
Below, we present the gradient complexity of Katalyst (i.e., the order of number of evaluations of $\nabla \phi_i(\x))$ based on the following basic assumptions. 
\begin{assumption}\label{a1} For problem~(\ref{eqn:P}), we assume that (i) $f_i(\cdot)$ is $L$-smooth and $\mu$-weakly convex,  (ii) $\psi$ is a non-smooth convex function, and (iii) there exists $\Delta_\phi>0$ such that $\phi(\x_0)-\min_{\x}\phi(\x)\leq \Delta_\phi$.
\end{assumption}

\begin{theorem}\label{thm:katyusha}
Suppose Assumption~\ref{a1} holds. Let $w_s = s^\alpha, \alpha>0, \gamma = \frac{1}{2\mu}$, $\Lh = L+\mu$, $\sigma = \mu$, and in each call of \text{Katyusha} let $\tau_1 = \min\{\sqrt{\frac{n\sigma}{3\Lh}},\frac{1}{2}\}$, step size $\eta = \frac{1}{3\tau_1\Lh}$, $\tau_2 = 1/2$, $\theta = 1+\eta\sigma$, and \[
K_s = \left\lceil\frac{\log (D_s)}{m\log (\theta)}\right\rceil,\quad m=\left\lceil\frac{\log(2\tau_1+2/\theta - 1)}{\log\theta}\right\rceil + 1,\] where $D_{s} = \max\{24\hat L/\mu, 2\hat L^3/\mu^3, 8L^2s/\mu^2\}$.
Then we have that
\begin{align*}
&\max\{\E[\|\nabla\phi_\gamma(\x_{\tau+1})\|^2], \E[L^2\|\x_{\tau+1} - \z_{\tau+1}\|^2]\}\\
&\leq \frac{34\mu\Delta_\phi(\alpha+1)}{S+1} + \frac{48\mu \Delta_\phi (\alpha+1)}{(S+1)\alpha^{\I_{\alpha<1}}},
\end{align*}
where  $\z_{\tau+1}= \text{prox}_{\gamma\phi}(\x_{\tau})$, $\tau$ is randomly chosen from $\{0, \ldots, S\}$ according to probabilities $p_\tau = \frac{w_{\tau+1}}{\sum_{k=0}^{S}w_{k+1}}, \tau=0, \ldots, S$.
Furthermore, the total gradient complexity for finding $\x_{\tau+1}$ such that \[\max(\E[\|\nabla\phi_\gamma(\x_{\tau+1})\|^2], L^2\E[\|\x_{\tau+1} - \z_{\tau+1}\|^2])\leq \epsilon^2\] is
\begin{align*}
N(\epsilon)=\left\{
    \begin{aligned}
       O\bigg( (\mu n+\sqrt{n\mu L})\log \bigg(\frac{L}{\mu\epsilon}\bigg)\frac{1}{\epsilon^2}\bigg),\;\; n\geq \frac{3L}{4\mu},\\
       O\bigg(\sqrt{nL\mu}\log \bigg(\frac{L}{\mu\epsilon}\bigg)\frac{1}{\epsilon^2}\bigg),\qquad n\leq \frac{3L}{4\mu}.
    \end{aligned}
\right.
\end{align*}
\end{theorem}
Indeed,  when $\psi=0$  we can derive a slightly stronger result stated in the following theorem. 
\begin{theorem}\label{cor:katyusha}
Suppose Assumption~\ref{a1} holds and $\psi=0$. With the same parameter values    as in Theorem \ref{thm:katyusha}  except that $K = \left\lceil\frac{\log (D)}{m\log (\theta)}\right\rceil$, where $D = \max(24\hat L/\mu, 2\hat L^3/\mu^3)$. The total gradient complexity for finding $\x_{\tau+1}$ such that $\E[\|\nabla\phi(\x_{\tau+1})\|^2]\leq \epsilon^2$ is
\begin{align*}
N(\epsilon)=\left\{
    \begin{aligned}
       O\bigg( (\mu n+\sqrt{n\mu L})\log \bigg(\frac{L}{\mu}\bigg)\frac{1}{\epsilon^2}\bigg),\;\; n\geq \frac{3L}{4\mu},\\
       O\bigg(\sqrt{nL\mu}\log \bigg(\frac{L}{\mu}\bigg)\frac{1}{\epsilon^2}\bigg),\qquad n\leq \frac{3L}{4\mu}.
    \end{aligned}
\right.
\end{align*}
\end{theorem}
{\bf Remark:} Our results in the above two theorems match that in~\citep{DBLP:journals/corr/abs/1805.05411}. Indeed, our result in Theorem~\ref{thm:katyusha} is slightly more general than that in~\citep{DBLP:journals/corr/abs/1805.05411}, which only considers the constrained smooth optimization with $\psi$ being the indicator function of a convex set. 


\subsection{Analysis}
In this subsection, we will present the convergence analysis for Katalyst. 
We first state the convergence property of modified Katyusha (Algorithm~\ref{alg:kat})
for solving following problem:
\begin{align}\label{eqn:P kat}
\min_{\x\in\R^d} f(\x)\coloneqq \hat{f}(\x) + \hat{\psi}(\x) = \frac{1}{n}\sum_{i=1}^n\hat{f}_i(\x) + \hat{\psi}(\x),
\end{align}
where each $\hat{f}_i $ is $\hat{L}$-smooth and convex, $\hat{\psi}(\x)$ is $\sigma$-strongly convex.
\begin{theorem}{(One call of Katyusha)}\label{thm:katyusha0}
Suppose that $\tau_1 = \min\{\sqrt{\frac{n\sigma}{3\hat{L}}},\frac{1}{2}\}$, $\tau_2 = 1/2$, $\eta = \frac{1}{3\tau_1 \hat{L}}$,and $m=\lceil\frac{\log(2\tau_1+2/\theta - 1)}{\log\theta}\rceil + 1$. Defining $\theta \coloneqq 1+\eta\sigma$, $D_t \coloneqq f(\y_t)-f(\x)$, $\Dt^k\coloneqq f(\widetilde{\x}^k) - f(\x)$ for any $\x$, Algorithm~\ref{alg:kat} outputs a solution $\widetilde{\x}^K$ of problem~(\ref{eqn:P kat}) such that
\begin{align}\label{eqn:key}
    \E[\Dt^K] &\leq 2\tau_1\theta^{-mK}(\frac{1-\tau_1}{\tau_1}\Dt^0 + \frac{1}{2\eta}\|\zeta_0-\x\|^2).
\end{align}
\end{theorem}
The proof of above theorem is deferred to Appendix~\ref{app:thm katyusha0}.

\begin{proof}{[of Theorem~\ref{thm:katyusha}]}
    Given Thoerem~\ref{thm:katyusha0}, our analysis is divided into several parts. First, we verify the value of $K$ is a valid one. Then, we apply the above theorem to show the convergence for solving each constructed function $f_s$. Then, we prove the convergence of $\|\nabla\phi_\gamma(\x_{\tau+1})\|$, followed by the convergence analysis of $L\|\x_{\tau+1} - \z_{\tau+1}\|$. Then, we briefly prove Theorem~\ref{cor:katyusha}. Finally, we derive the gradient complexity. 
\paragraph{Validation of $K$:}
Overall, we need 
\begin{align*}
\theta^{-mK}\leq  \min\left\{\frac{\mu}{24\hat L}, \frac{\mu^3}{2\hat L^3}, \frac{\mu^2}{8L^2s}\right\}.
\end{align*}
Define $D_{s} = \max\{24\hat L/\mu,  2\hat L^3/\mu^3, 8L^2s/\mu^2\}\geq 16$. 
We can set $K = \lceil\frac{\log (D_{\max})}{m\log \theta}\rceil$. Then, 

\begin{align*}
    K \geq  \bigg\lceil\frac{4}{m\log\theta}\bigg\rceil\geq 1,
\end{align*}
where the last inequality follows that  $2/\theta\geq 2\tau_1+2/\theta-1\geq 1$ always hold according to the setting of $\tau_1 = \min\{\sqrt{\frac{n\mu}{3\Lh}},\frac{1}{2}\}$ and $\eta = \frac{1}{3\tau_1\Lh}$.

\paragraph{Convergence of $\|\nabla\phi_\gamma(\cdot)\|$.} 
Let $\z_s = \arg\min_{\x}f_s(\x)$ and $\x_*$ denote the global minimum of $\min_{\x}\phi(\x)$. It is notable that $\|\x_{s-1} - \z_s\|/\gamma = \nabla\phi_\gamma(\x_{s-1})$.  Below, we will use $K$ to denote $K_s$. $\E_s$ denotes the expectation over randomness in the $s$-th stage conditioned on all previous stages. 
Applying Theorem~\ref{thm:katyusha0} to the $s$-th call of Katyusha, we have
\begin{align}\label{eqn:k2}
\E_s[f_s(\x_s) - f_s(\z_s)] \leq& 4\theta^{-mK}(f_s(\x_{s-1}) - f_s(\z_s)) \nonumber \\
&+ \frac{2\tau_1\theta^{-mK}}{\eta}\|\x_{s-1} - \z_s\|^2.
\end{align}
It is easy to see that 
\begin{align*}
&f_s(\x_{s-1}) - f_s(\z_s)) \\
= &\phi(\x_{s-1}) - \phi(\z_s) - \frac{1}{2\gamma}\|\x_{s-1} -\z_s\|^2\nonumber\\
\leq&  \phi(\x_{s-1}) - \phi(\x_*) - \frac{1}{2\gamma}\|\x_{s-1} -\z_s\|^2.
\end{align*}
Thus, we have
\begin{align*}&\E_s[f_s(\x_s) - f_s(\z_s)] \nonumber\\
\leq&  4\theta^{-mK}(\phi(\x_{s-1}) - \phi(\x_*)) + \frac{2\tau_1\theta^{-mK}}{\eta}\|\x_{s-1} - \z_s\|^2\nonumber\\
\leq& \underbrace{4\theta^{-mK}(\phi(\x_{s-1}) - \phi(\x_*)) + 2\theta^{-mK}\hat L\|\x_{s-1} - \z_s\|^2}\limits_{\mathcal E_s}.
\end{align*}
Based on the above result and by utilizing the strong convexity of $f_s$ and simple algebra, we have the following result whose proof is in Appendix~\ref{app:lem1}. 
\begin{lemma}\label{lem:1}
Let $\Delta_s = \phi(\x_{s-1}) - \phi(\x_s)$ and $\theta^{-mK}\leq \mu/(24\hat L)$. Then we have that 
\begin{align*}
\frac{1}{8\gamma}\|\x_{s-1}-\z_s\|^2\leq \E_{s}[\Delta_s] + 12\theta^{-mK}(\phi(\x_{s-1}) - \phi(\x_*)).
\end{align*} 
\end{lemma}
It implies that 
\begin{align*}
\|&\nabla\phi_\gamma(\x_{s-1})\|^2 \nonumber\\
\leq& \E_{s}[8\Delta_s/\gamma] + 96\theta^{-mK}(\phi(\x_{s-1}) - \phi(\x_*))/\gamma.
\end{align*} 
Multiplying both sides by $w_s$, we have that
\begin{align*}
  & w_s\E_s[\|\nabla \phi_{\gamma}(\x_{s-1})\|^2]\\
  \leq& \E_s\bigg[8w_s\Delta_{s}/\gamma  + 96\theta^{-mK}w_s(\phi(\x_{s-1}) - \phi(\x_*))/\gamma\bigg].
    \end{align*}
 By summing over $s=1,\ldots,S+1$, we have
 \begin{align*}
&\E[ \sum_{s=1}^{S+1}w_s \|\nabla \phi_{\gamma}(\x_{s-1})\|^2]\\
\leq& \E\bigg[\frac{8}{\gamma}\sum_{s=1}^{S+1}w_s\Delta_{s} + \frac{96}{\gamma}\sum_{s=1}^{S+1}w_s\theta^{-mK}(\phi(\x_{s-1}) - \phi(\x_*)) \bigg].
   \end{align*}
Taking the expectation w.r.t. $\tau\in \{0,\ldots, S\}$, we have that
\begin{align*}
   \E[\|\nabla\phi_{\gamma}(\x_{\tau})\|&^2]] \leq \E\bigg[\frac{8\sum_{s=1}^{S+1}w_s\Delta_s}{\gamma\sum_{s=1}^{S+1}w_s} \\
   &+ \frac{96\sum_{s=1}^{S+1}w_s\theta^{-mK}(\phi(\x_{s-1}) - \phi(\x_*))}{\gamma\sum_{s=1}^{S+1}w_s}\bigg].
\end{align*}
Next, we bound the numerators of the two terms in the above bound. For the first term in the above bound, we use Lemma~\ref{lem:2} in the Appendix~\ref{app:lem2} and have 
$\E\left[\sum_{s=1}^{S+1} w_s\Delta_s \right]\leq  \Delta_\phi w_{S+1}$. We can bound the second term as following:
\begin{align*}
   &\E\bigg[ \sum_{s=1}^{S+1}w_s\theta^{-mK}(\phi(\x_{s-1}) - \phi(\x_*))\bigg]\\
   \leq&\sum_{s=1}^{S+1}w_s\theta^{-mK}\E[\phi(\x_{s-1}) - \phi(\x_*)]
   \leq \Delta_\phi\sum_{s=1}^{S+1}w_s\theta^{-mK},
\end{align*}
where we use the fact $\E[\phi(\x_s)]\leq \E[\phi(\x_{s-1})]$ as shown in the proof of Lemma~\ref{lem:2}. 
As a result, 
\begin{align*}
   &\E[\|\nabla\phi_{\gamma}(\x_{\tau})\|^2]] \\
   \leq& \bigg[\frac{8\Delta_\phi w_{S+1}}{\gamma\sum_{s=1}^{S+1}w_s} + \frac{96\Delta_\phi\sum_{s=1}^{S+1}w_s\theta^{-mK}}{\gamma\sum_{s=1}^{S+1}w_s}\bigg]\\
   \leq& \bigg[\frac{8\Delta_\phi w_{S+1}}{\gamma\sum_{s=1}^{S+1}w_s} + \frac{12\Delta_\phi\sum_{s=1}^{S+1}w_ss^{-1}}{\gamma\sum_{s=1}^{S+1}w_s}\bigg],
\end{align*}
where we use the fact $\theta^{-mK}\leq 1/(8s)$.  Then by simple algebra (cf.~\citep{chen18stagewise}), we have 
\[
\E[\|\nabla\phi_\gamma(\x_{\tau})\|^2]\leq \frac{16\mu\Delta_\phi (\alpha+1)}{S+1} + \frac{24\mu\Delta_\phi (\alpha+1)}{(S+1)\alpha^{\mathbb I_{\alpha<1}}}.
\]
Due to the objective decreasing property, we have
\begin{align*}
\E[\phi(\x_s) + \frac{1}{2\gamma}\|\x_s - \x_{s-1}\|^2 - \phi(\x_{s-1})]\leq 0,
\end{align*}
which implies by a similar analysis
\begin{align*}
\frac{1}{2\gamma}\E[\|\x_{\tau+1} - \x_{\tau}\|^2]\leq  \frac{\Delta_\phi(\alpha+1)}{S+1}.
\end{align*}
Since $\phi_\gamma(\x)$ has  $(\gamma^{-1} - \mu)$-Lipschitz continuous gradient (cf. Lemma 2.1 in~\citep{Drusvyatskiy2018}), then we have
\begin{align*}
&\E[\|\nabla\phi_\gamma(\x_{\tau+1})\|^2]\\
\leq& 2\E[\|\nabla\phi_\gamma(\x_{\tau})\|^2] + 2(\gamma^{-1} - \mu)^2\E[\|\x_{\tau+1} - \x_\tau\|^2]\\
\leq&  2\E[\|\nabla\phi_\gamma(\x_{\tau})\|^2] + \frac{2\mu\Delta_\phi(\alpha+1)}{S+1}\\
\leq&  \frac{34\mu\Delta_\phi (\alpha+1)}{S+1}  + \frac{48\mu\Delta_\phi (\alpha+1)}{(S+1)\alpha^{\mathbb I_{\alpha<1}}}.
\end{align*}

\paragraph{Convergence of $L\|\x_{\tau+1} - \z_{\tau+1}\|$.} By the strong convexity of $f_s$, we have $\E[\|\x_{s} - \z_s\|^2] \leq \frac{2}{\sigma}\mathcal E_s$. 
To proceed,   we have
\begin{align*}
&L^2\|\x_{s} - \z_{s}\|^2\leq \frac{2L^2}{\sigma}(4\theta^{-mK}(\phi(\x_{s-1}) - \phi(\x_*))\\
& +2\theta^{-mK}\hat L\|\x_{s-1} - \z_{s}\|^2 )\\
\leq& \frac{8L^2\theta^{-mK}}{\sigma}(\phi(\x_{s-1}) - \phi(\x_*)) \\
& +\frac{2\hat L^3\theta^{-mK}}{\mu^3}\|\nabla\phi_\gamma(\x_{s-1})\|^2\\
\leq& \frac{8L^2\theta^{-mK}}{\sigma}(\phi(\x_{s-1}) - \phi(\x_*))  + \|\nabla\phi_\gamma(\x_{s-1})\|^2,
\end{align*}
where we use the fact $\|\x_{s-1} - \z_s\|/\gamma = \|\nabla\phi_\gamma(\x_{s-1})\|$ and $\theta^{-mK}\leq \mu^3/(2\hat L^3)$. 
Then following the same analysis as above, 
\begin{align*}
&\E[L^2\|\x_{\tau+1} - \z_{\tau+1}\|^2]\\
\leq& \frac{8L^2\Delta_\phi\sum_{s=1}^{S+1}w_s\theta^{-mK}}{\sigma\sum_{s=1}^{S+1}w_s} +\E[\|\nabla\phi_\gamma(\x_{\tau})|^2].
\end{align*}
Since $\theta^{-mK}\leq \mu^2/(8L^2s)$, then
\begin{align*}
\E[L^2\|\x_{\tau+1} - \z_{\tau+1}\|^2]\leq& \frac{16\mu\Delta_\phi(\alpha+1)}{S+1}\\
& + \frac{25\mu \Delta_\phi (\alpha+1)}{(S+1)\alpha^{\I_{\alpha<1}}}.
\end{align*}
When  $\psi(\cdot)=0$ and considering $\alpha$ as a constant, we have
\begin{align*}
&\E[\|\nabla \phi(\x_{\tau+1})\|^2]\\
\leq & \E[\|\nabla \phi(\x_{\tau+1})-\nabla \phi(\z_{\tau+1})+\nabla \phi(\z_{\tau+1})\|^2] \\
\leq& \E[2L^2\|\x_{\tau+1}-\z_{\tau+1}\|^2+2\|\nabla\phi_\gamma(\x_{\tau})\|^2]\\
\leq & O\left(\frac{\mu\Delta_\phi}{S+1}\right).
\end{align*}
Indeed,  for $\psi(\cdot)=0$, we can do slightly better by bounding $f_s(\x_{s-1}) - f_s(\z_s)\leq \frac{\hat L}{2}\|\x_{s-1} - \z_s\|^2$. Then $\mathcal E_s$ becomes $2\theta^{-mK}\hat L\|\x_{s-1} - \z_s\|^2$ and $\theta^{-mK}(\phi(\x_{s-1}) - \phi(\z_s))$ in the proceeding analysis is gone, which removes the requirement $\theta^{-mK}\leq \mu^2/(16L^2s)$. As a result, we can set $K = \lceil \log(D)/(m\log\theta) \rceil$, where $D = \max(48L/\mu, 2\hat L^3/\mu^3)$.


\paragraph{Gradient Complexity: } Finally, we analyze the gradient complexity. Let us consider the gradient complexity at the $s$-th stage, which is 
{\begin{align*}
(n+m)K \leq& \frac{2\log(D_s)}{\log(2\tau_1+2\tau_1\eta\mu+1-\eta\mu)}n \\
&+ \frac{2\log(D_s)}{\log(1+\eta\mu)}.
\end{align*}}
Let $\tau_1 = \frac{c}{\eta\mu}$, where $0\leq c = \frac{\mu}{3\Lh}\leq \frac{1}{3}$. We have that
\begin{align*}
    &(n+m)K = nK+mK \\
    =& \frac{2\log(D_s)}{\log(2\tau_1+2\tau_1\eta\mu+1-\eta\mu)}n + \frac{2\log(D_s)}{\log(1+\eta\mu)}\\
    \leq& \frac{2\log(D_s)}{\log( 2\tau_1+2c+1-\frac{c}{\tau_1})}n + \frac{2\log(D_s)}{\log(1+\frac{c}{\tau_1})}.
\end{align*}
We analyze  two cases. 

\textbf{Case 1:} If $n\geq \frac{3\Lh}{4\mu}$, then $\tau_1=\frac{1}{2}$, we have that
\begin{align*}
    (n+m)K &\leq O\bigg(\log(D_s)n + \frac{\log D_s}{\log(1+2c)}\bigg).
\end{align*}
{
Since $0\leq c\leq 1/3$ so $\log(1+2c)\geq c$, then 
\begin{align*}
    (n+m)K &\leq O\bigg((n+\frac{\Lh}{\mu})\log D_s\bigg).
\end{align*}
Then the total gradient complexity for finding $\E\|\nabla\phi_\gamma(\x_\tau)\|^2\leq \epsilon^2$ is $ O((\mu n+ L)\log(L/(\mu\epsilon)).$
}
\begin{figure*}[ht]
\begin{center}

    {\includegraphics[scale=.21]{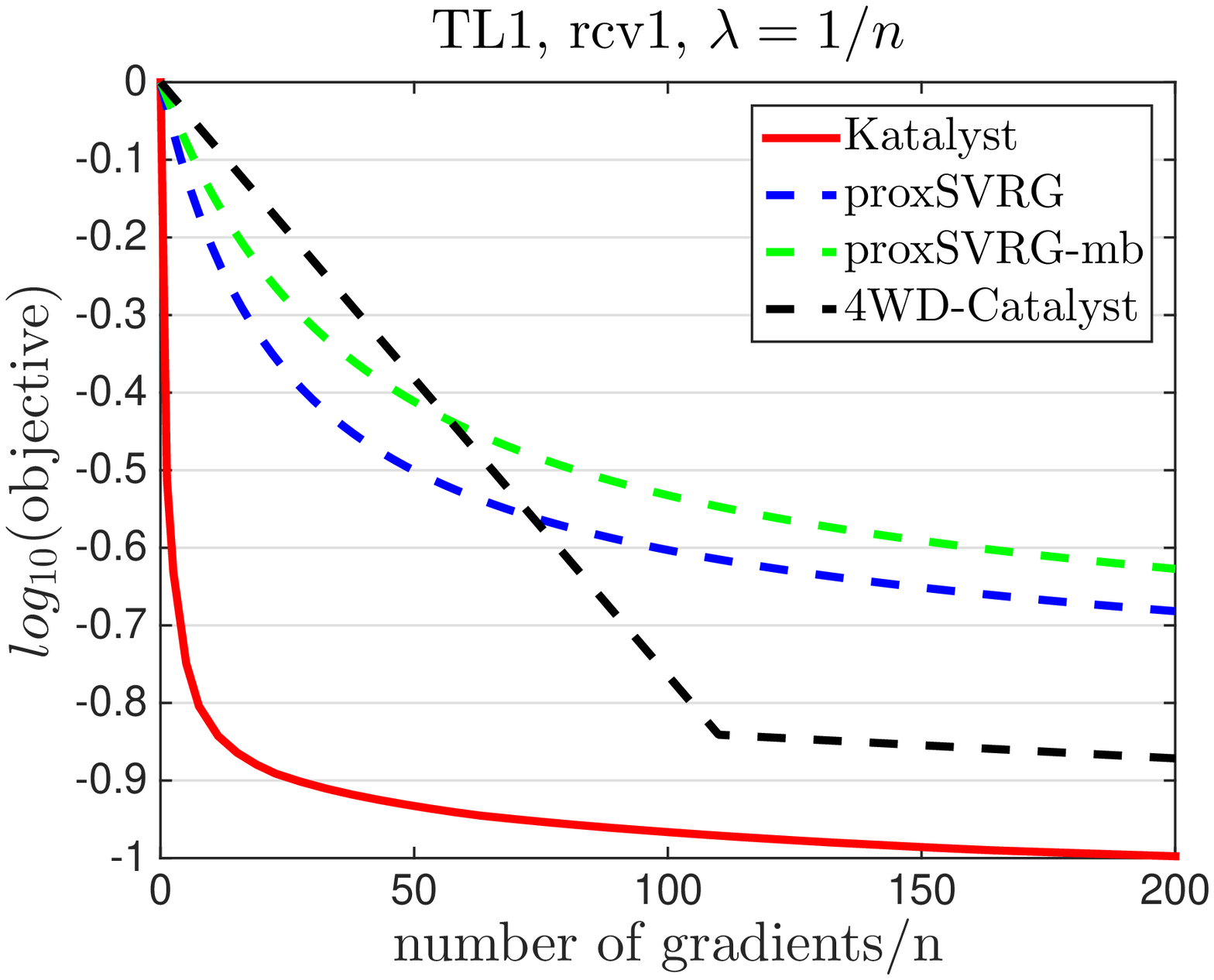}}
    {\includegraphics[scale=.21]{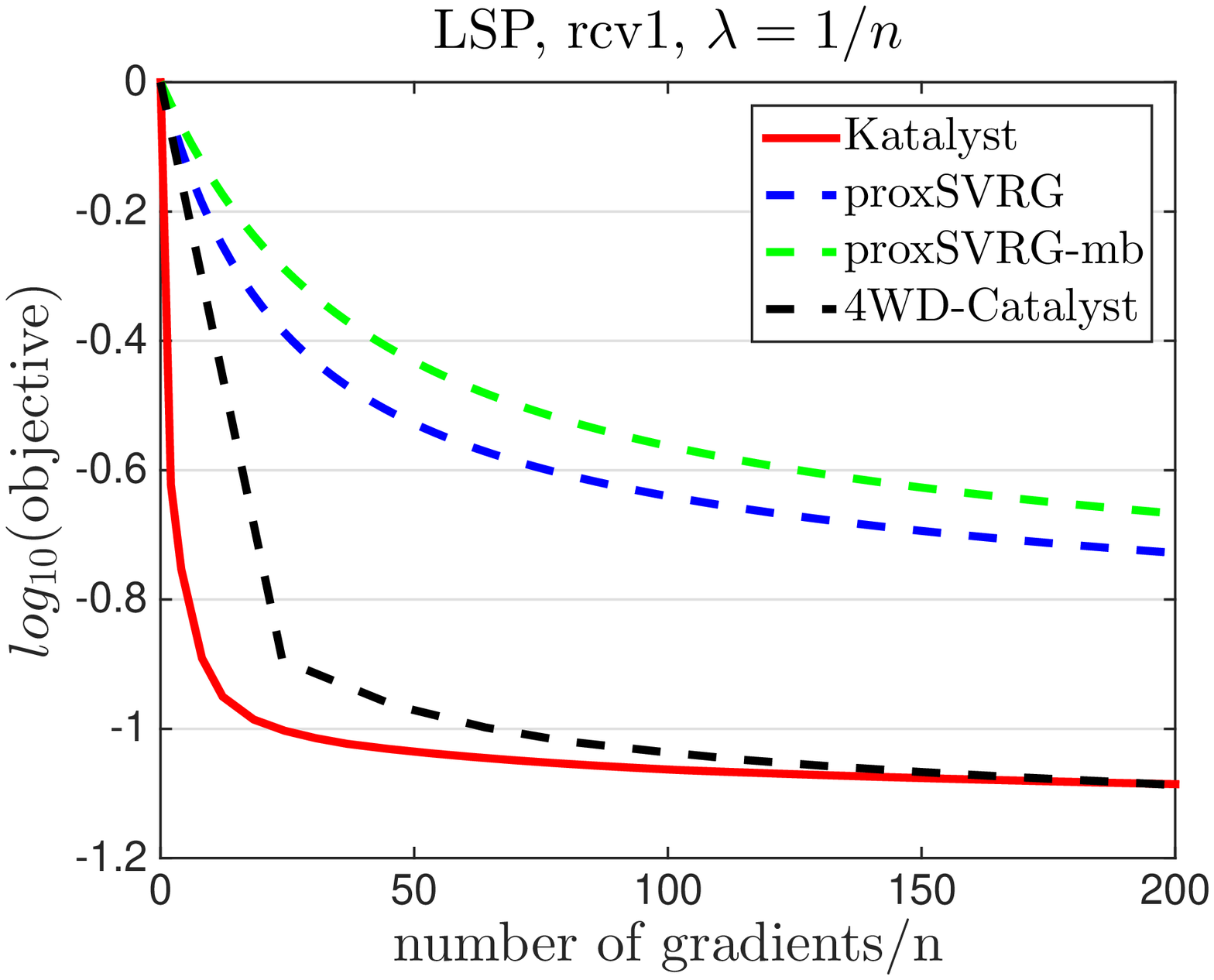}}
    {\includegraphics[scale=.21]{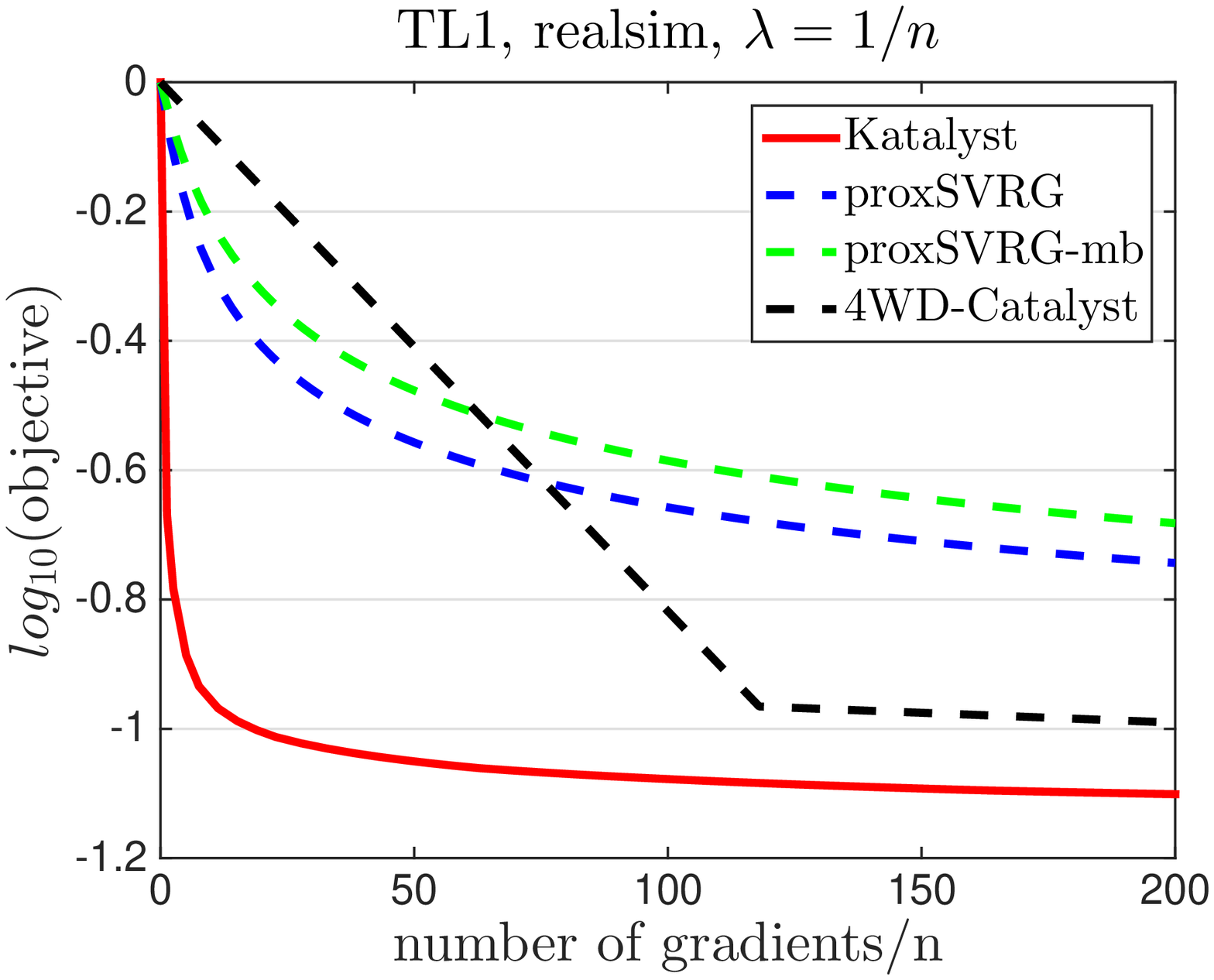}}
    {\includegraphics[scale=.21]{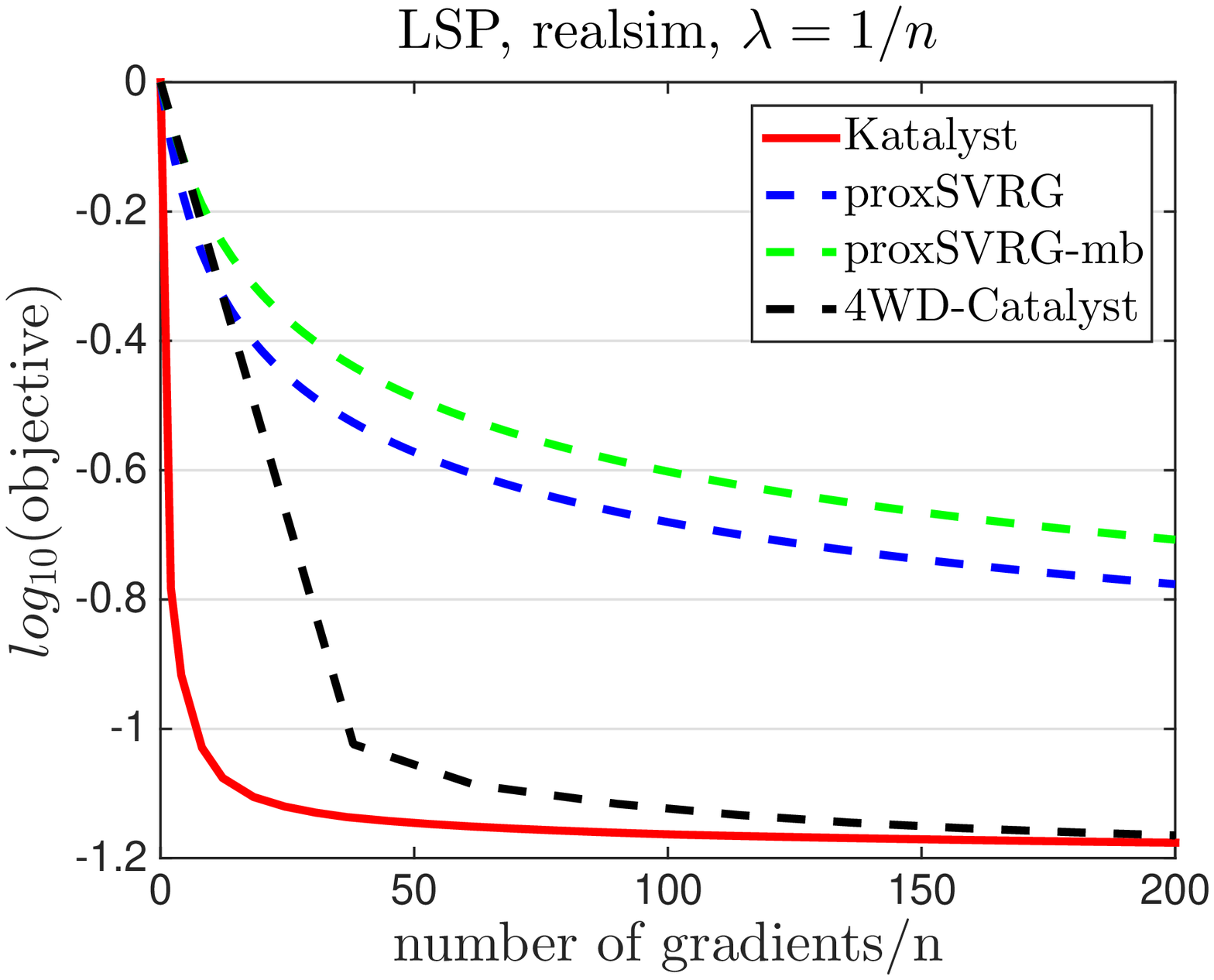}}
    {\includegraphics[scale=.21]{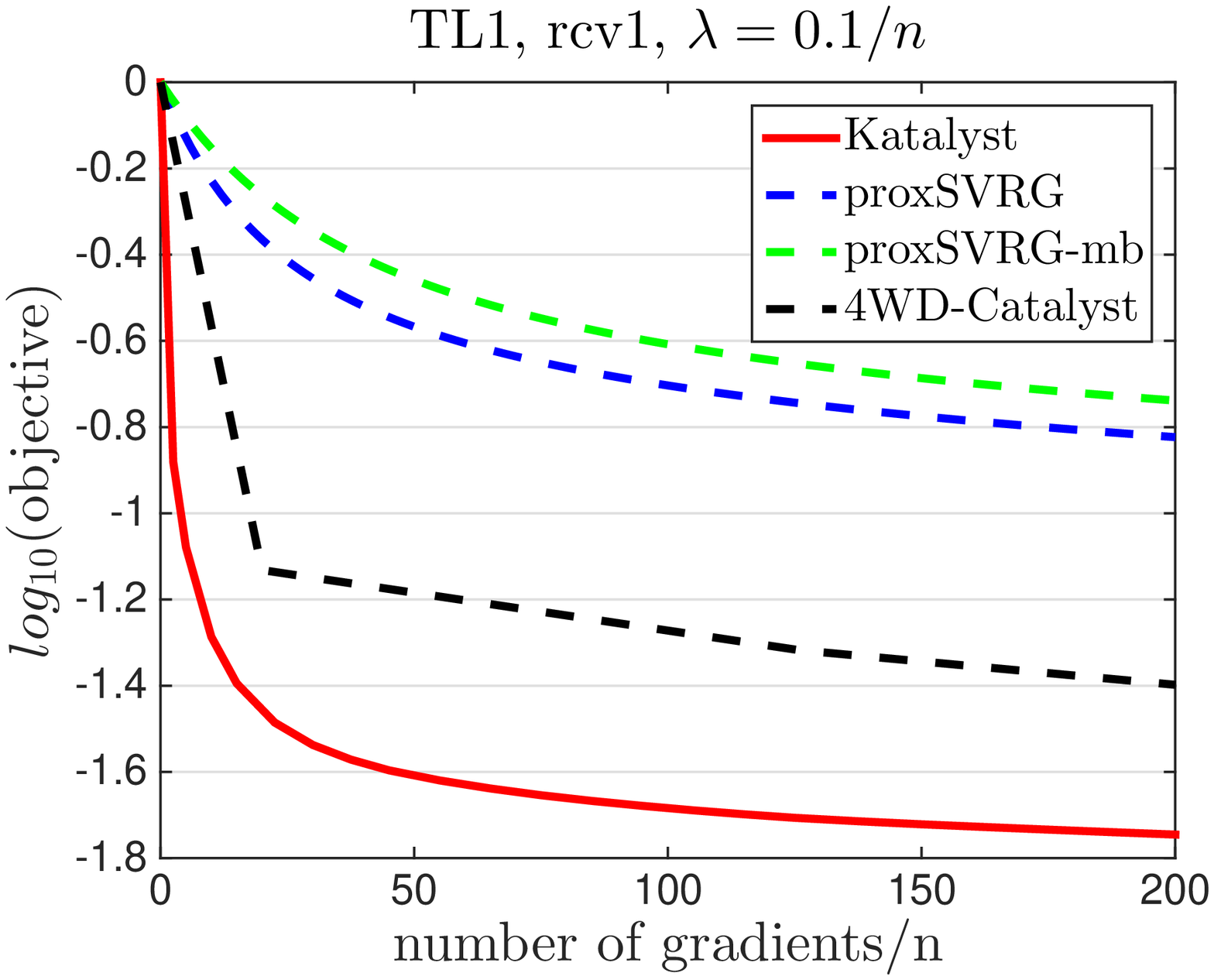}}
    {\includegraphics[scale=.21]{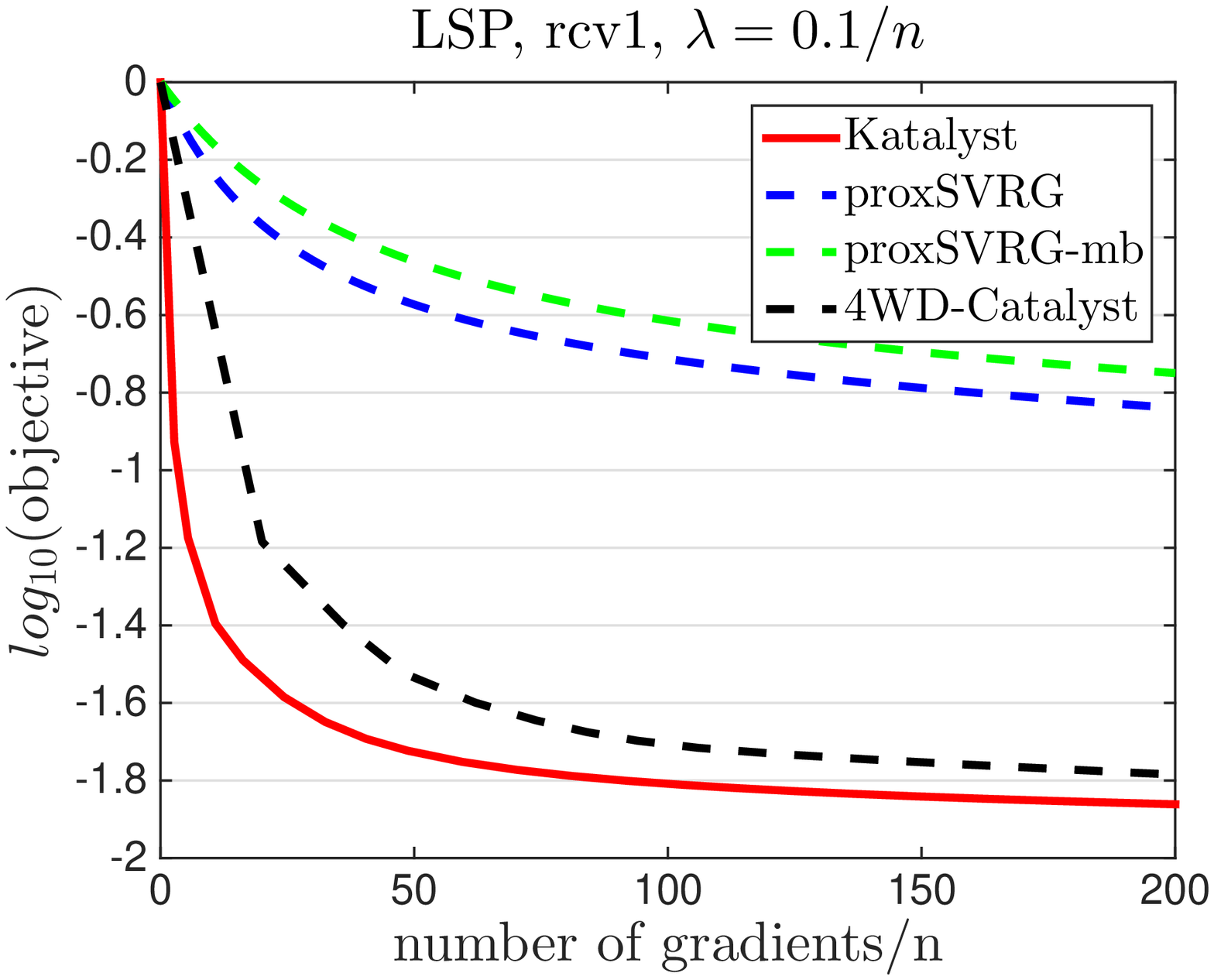}}
    {\includegraphics[scale=.21]{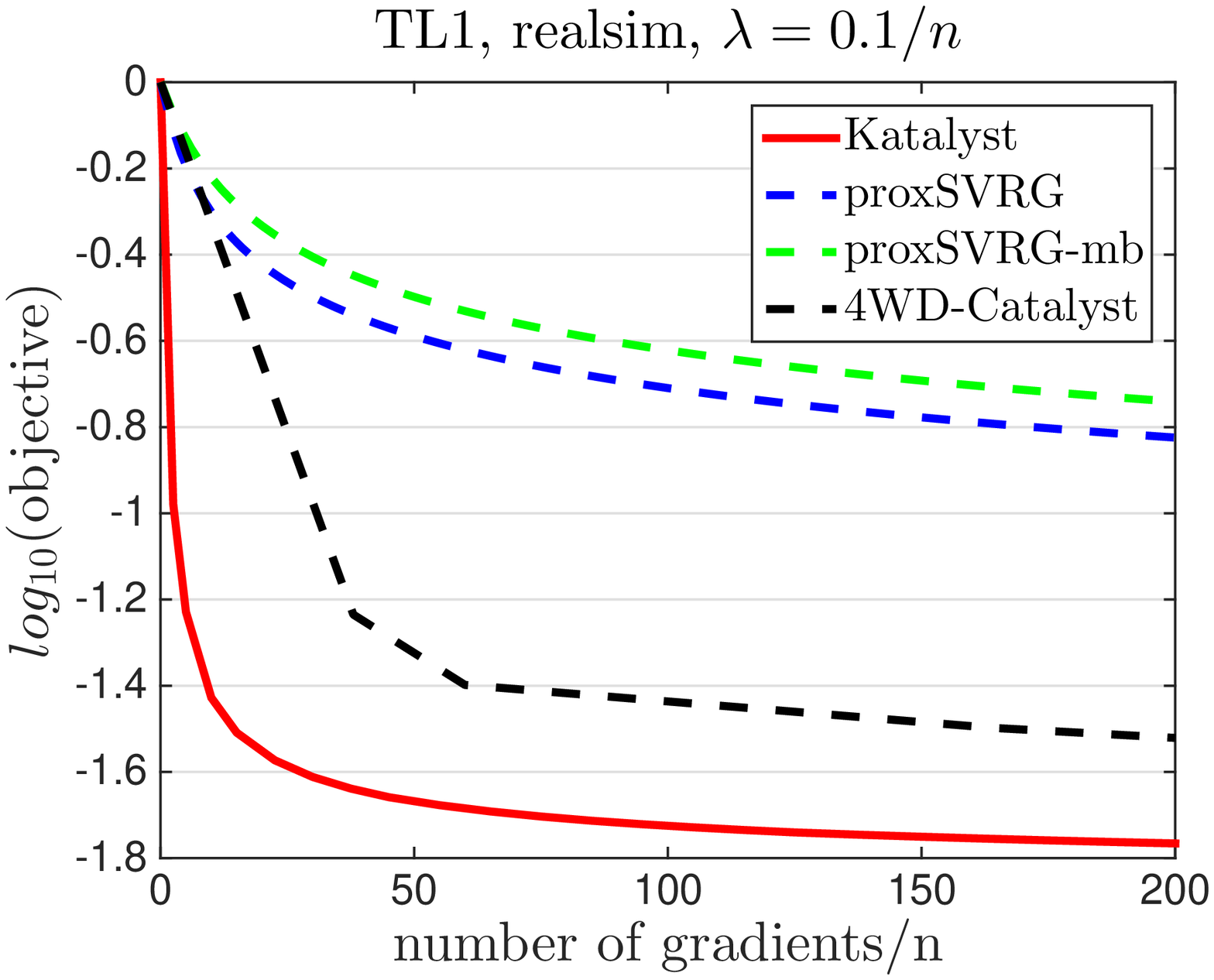}}
    {\includegraphics[scale=.21]{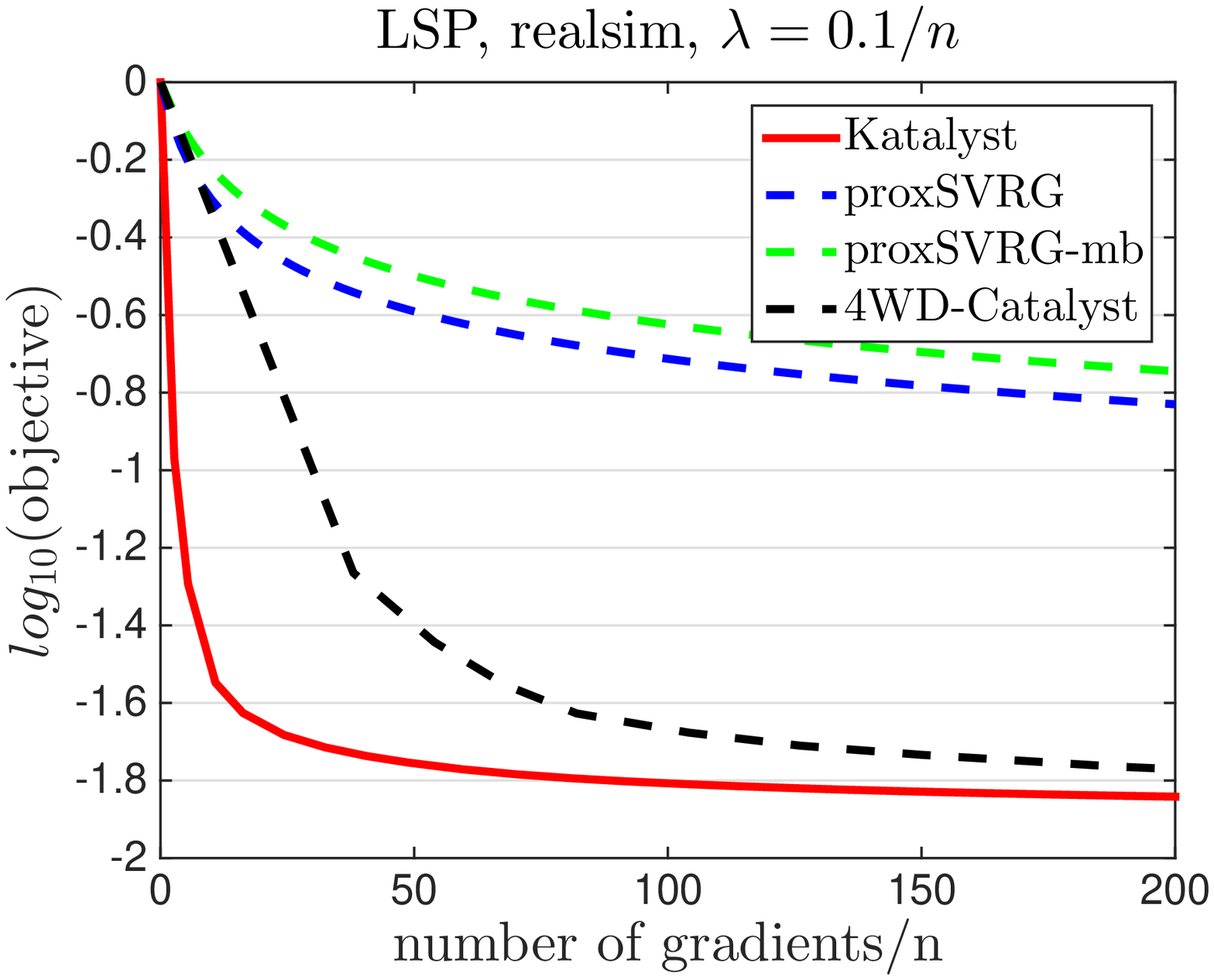}}

    \caption{Comparison of different algorithms for two tasks on different datasets}
    \label{fig1}
    \vspace*{-0.1in}
\end{center}
\end{figure*}

 \textbf{Case 2:} If $n\leq \frac{3\Lh}{4\mu}$,
then  $\tau_1 = \sqrt{\frac{n\mu}{3\Lh}}\in(0,\frac{1}{2}]$. We have that following inequalities hold
\begin{align*}
\frac{\tau_1}{c} = \sqrt{\frac{3n\Lh}{\mu}}&\leq \frac{3\Lh}{2\mu}, \quad\log(\frac{\tau_1+c}{c})\leq  \log \frac{3\Lh}{\mu},\\
& \log(1+ c/\tau_1)\geq  c/(2\tau_1)
\end{align*}
and due to 
\item if $2\tau_1+2c-\frac{c}{\tau_1}\leq 1/2$, then
\begin{align*}
    \log(2\tau_1+2c+1-\frac{c}{\tau_1})\geq \tau_1 + c - c/(2\tau_1),
\end{align*}
\item if $2\tau_1+2c-\frac{c}{\tau_1}\geq  1/2$, then
\begin{align*}
    \log(2\tau_1+2c+1-\frac{c}{\tau_1})\geq \log(1.5),
\end{align*}
 we have 
 \begin{align*}
 &\frac{1}{\log(2\tau_1+2c+1-\frac{c}{\tau_1})}\\
 \leq& \max\bigg\{\frac{1}{\log(1.5)}, \frac{1}{\sqrt{\frac{n\mu}{3\Lh}}+\frac{\mu}{3\Lh}-\sqrt{\frac{\mu}{12n\Lh}}}\bigg\}\\
 \leq& O\bigg(\sqrt{\frac{\Lh}{n\mu}}\bigg).
 \end{align*}
Thus we have 
\begin{align*}
(n+m)K\leq O\bigg(\sqrt{\frac{n\Lh}{\mu}}\log\frac{\Lh}{\mu}\bigg),
\end{align*}
and the total gradient complexity for finding $\E\|\nabla\phi_\gamma(\x_\tau)\|^2\leq \epsilon^2$ is $ O(\sqrt{\mu n L }\log(L/(\mu\epsilon)).$

\end{proof}

\section{Experiments}
In this section, we conduct some experements for solving regularized classification problem in the form of~(\ref{eqn:ncsr}) with  $\ell(\x;\a_i,b_i) = \frac{1}{2}(\max(0,1-b_i\a_i^\top\x))^2$ being a squared hinge loss that is more suitable for classification. 


\textbf{Penalties and Parameters.} We choose two different non-convex and non-smooth penalty functions as the regularizers, namely log-sum penalty (LSP) $R(\x) =\sum_{i=1}^d\log(\beta + |x_i|), (\beta > 0)$ and transformed $\ell_1$ (TL1) penalty $R(\x) = \sum_{i=1}^d\frac{(\beta +1) |x_i|}{\beta + |x_i|}, (\beta > 0)$ where $\beta>0$ is a parameter. 
Both LSP and TL1 can be written as a difference of convex functions: $R(\x) =   r_1(\x) - r_2(\x)$, 
where $r_1(\x)$ is a scaled $\ell_1$ norm and $r_2(x)$ is smooth and convex (cf. details provided in Appendix~\ref{app:dec}). 
Then the problem becomes 
\begin{align*}
\min_{\x\in\R^d} \phi(\x)  := \frac{1}{n}\sum_{i=1}^{n}  \underbrace{(\ell(\x;\a_i, b_i)  - r_2(\x)) }\limits_{f_i(\x)} +    \underbrace{r_1(\x)}\limits_{\Psi(\x)} ,
\end{align*}
For LSP, it is easy to show that the weakly convexity parameter and smoothness parameter of $f_i(\x)$ are given by $\mu = \frac{\lambda}{\beta^2}~~\text{and}~~\hat L = \frac{\lambda }{\beta^2} + \max_{1\leq i \leq d}\|\a_i\|^2$. 
For TL1, it is easy to show that the weakly convexity parameter and smoothness parameter of $\hat f(\x)$ are given by 
$\mu = \frac{2(\beta+1)\lambda}{\beta^2}~~\text{and}~~\hat L =  \frac{2(\beta+1)\lambda}{\beta^2} + \max_{1\leq i \leq d}\|\a_i\|^2$.  We fix $\beta = 1$ but set two different values of  $\lambda \in \{ \frac{1}{n}, \frac{0.1}{n}\}$. 
The experiments are performed on two data sets from libsvm website \citep{CC01a}, namely rcv1 ($n=20,242$ and $d = 47,236$) and real-sim ($n = 72,309$ and $d  = 20,958$).

\textbf{Baselines and Settings.} We compare the proposed Katalyst with proxSVRG, its mini-batch variant (named proxSVRG-mb in experiments) \citep{reddi2016proximal} and 4WD-Catalyst \citep{pmlr-v84-paquette18a}. Other algorithms like RapGrad, SPIDER, SNVRG are not applicable to the considered problem. 
Since smoothness parameter $L$ and weak convexity parameter $\mu $ are given as discussed above, we implement Algorithm~1 in \citep{pmlr-v84-paquette18a} for 4WD-Catalyst. 
All parameters in three baselines including step size and the number of iterations for the inner loop  are set to their theoretical values suggested in the original papers.

\textbf{Results.} We report the results in Figure~\ref{fig1}, where the x-axis is $(\text{number of gradients})/n$ and the y-axis is log-scale of the objective value. 
For 4WD-Catalyst, we only plot the result at the end of each stage since it selects the better solution of two sub-problems.
It is worth noting that we do not include the complexity of computing $f_\kappa(\bar{\x}_s;\x_{s-1})$ in solving sub-problem for 4WD-Catalyst, i.e. \citep[eqn. (7) of Algorithm~1]{pmlr-v84-paquette18a}, which would introduce more CPU time in practice.

We can observe that when using a smaller $\lambda$ that gives a smaller value of  $\mu$-convexity parameter,  Katalyst has relatively larger speed-up compared with the two variants of proxSVRG, which supports the presented complexity of Katalyst that is adaptive to the weakly convex property. 
Katalyst is also more efficient than  4WD-Catalyst, which needs to solve two sub-problems at each stage to satisfy a certain criterion that requires many iterations in practice.

\section{Conclusion}
In this paper, we have developed a SVRG-type accelerated stochastic algorithm for solving a family of non-convex optimization problems whose objective consists of a finite-sum of smooth functions and a non-smooth convex function. We proved that the gradient complexity can be improved when the condition number is very large compared to the number of smooth components, which achieves the best complexity among all SVRG-type methods and also matches that of an existing SAGA-type stochastic algorithm. 

\bibliographystyle{icml2019}
\bibliography{all}

\appendix
\section{Proof of Theorem~3}\label{app:thm katyusha0}
We need the following lemma for proving Theorem~3. 
\begin{lemma}{\citep{DBLP:conf/stoc/Zhu17}}\label{katyusha}
Regarding the modified-Katyusha algorithm (Algorithm~2), suppose that $\tau_1\leq\frac{1}{3\eta \Lh}$, $\tau_2 = 1/2$. Defining $D_t:=f(\y_t)-f(\x), \; \Dt^k:=f(\widetilde{\x}^k) - f(\x)$ for any $\x$,  conditioned  on iterations $\{0,\ldots, t-1\}$ in $k$-th epoch and all iterations before $k$-th epoch,  we have that
\begin{align}\label{eqn:katyusha}
0\leq& \frac{(1-\tau_1-\tau_2)}{\tau_1} D_t - \frac{1}{\tau_1}\E[ D_{t+1}] + \frac{\tau_2}{\tau_1}\Dt^k \nonumber\\
&+ \frac{1}{2\eta} \|\zeta_t-\x\|^2 - \frac{1+\eta\sigma}{2\eta}\E[\|\zeta_{t+1}-\x\|^2]
\end{align}
\end{lemma}
\begin{proof}{[of Theorem~3]}
Define $\theta = 1+\eta\sigma$ and multiply (\ref{eqn:katyusha}) by $\theta^{t}$ on both side. By summing up the inequalities in~(\ref{eqn:katyusha})  in the $k$-th epoch, we have that
\begin{align*}
0\leq &\E_{k}\bigg[\frac{1-\tau_1-\tau_2}{\tau_1}\sum_{t=0}^{m-1}D_{km+t}\theta^t-\frac{1}{\tau_1}\sum_{t=0}^{m-1}D_{km+t+1}\theta^t\bigg] \\
&+ \frac{\tau_2}{\tau_1}\Dt^k\sum_{t=1}^{m-1}\theta^t+\frac{1}{2\eta}\|\zeta_{km}-\x\|^2\\
&-\frac{\theta^m}{2\eta}\E_{k+1}[\|\zeta_{(k+1)m}-\x\|^2]
\end{align*}
where $\E_{k}[\cdot]$ denotes expectation in $k$-th epoch conditional on $0,\ldots, k-1$ epochs. Using the convexity of $f(\cdot)$, we have that 
\begin{align}\label{eqn:katyusha 1}
\nonumber &\frac{\tau_1+\tau_2-1+1/\theta}{\tau_1}\theta\E_{k}[\Dt^{k+1}]\sum_{t=0}^{m-1}\theta^t+\\
\nonumber & \;  \frac{1-\tau_1-\tau_2}{\tau_1}\theta^m\E[D_{(k+1)m}]+ \frac{\theta^m}{2\eta} \E_k[\|\zeta_{(k+1)m}-\x\|^2]\\
&\leq \frac{\tau_2}{\tau_1}\Dt^k \sum_{t=0}^{m-1}\theta^t + \frac{1-\tau_1-\tau_2}{\tau_1}D_{km} + \frac{1}{2\eta}\|\zeta_{km}-\x\|^2
\end{align}
Substituting $\tau_2 = 1/2$ and $m\leq \lceil\frac{\log(2\tau_1+2/\theta - 1)}{\log\theta}\rceil + 1$,   we have that 

\begin{align*}
    \nonumber &\theta^m\frac{1}{2\theta\tau_1}\E_{k}[\Dt^{k+1}]\sum_{t=0}^{m-1}\theta^t + \frac{1/2-\tau_1}{\tau_1}\theta^m \E[D_{(k+1)m}]\\
    &\; + \frac{\theta^m}{2\eta} \E_k[\|\zeta_{(k+1)m}-\x\|^2]\\
    &\leq \frac{1}{2\tau_1}\Dt^{k} \sum_{t=0}^{m-1}\theta^t + \frac{1/2-\tau_1}{\tau_1}D_{km} + \frac{1}{2\eta}\|\zeta_{km}-\x\|^2
\end{align*}
Telescoping above inequality over all epochs $k = 0,\ldots, K-1$ we have that
\begin{align*}
\E[\Dt^K] \leq 2\theta\tau_1\theta^{-mK}&\left( \frac{1}{2\tau_1}\Dt^0 + \frac{1/2-\tau_1}{\tau_1\sum_{t=0}^{m-1}\theta^t}D_0 \right.\\
&\;\left.+ \frac{1}{2\eta\sum_{t=0}^{m-1}\theta^t}\|\zeta_0-\x\|^2\right)
\end{align*}
Since $\sum_{t=0}^{m-1}\theta^t\geq 1$, $\tau_1\leq \frac{1}{2}$ and $\theta \leq 2$, we have 
\begin{align*}
    \E[\Dt^K] &\leq 4\tau_1\theta^{-mK}(\frac{1-\tau_1}{\tau_1}\Dt^0 + \frac{1}{2\eta}\|\zeta_0-\x\|^2)
\end{align*}
We can use the same analysis by plugging $\x= \zeta_0$ in~(\ref{eqn:katyusha}) to  prove that $\E[f(\widetilde\x^{K}) - f(\widetilde\x^0)]\leq 0$ - an objective value decreasing property that will be used later. 
\end{proof}

\section{Proof of Lemma~\ref{lem:1}}\label{app:lem1}
\begin{proof}
First we have hat
\begin{align*}
    \E[f_s(\x_{s})] =& \E\bigg[\phi(\x_s)) + \frac{1}{2\gamma} \|\x_s-\x_{s-1}\|^2\bigg] \leq f_s(\z_s) + \mathcal{E}_s \\
    \leq& f_s(\x_{s-1}) + \mathcal{E}_s = \phi(\x_{s-1}) + \mathcal{E}_s
\end{align*}
Besides, we also have that
\begin{align*}
    &\|\x_{s}-\x_{s-1}\|^2 \\
    =& \|\x_{s}-\z_{s}+\z_{s}-\x_{s-1}\|^2\\
    =& \|\x_{s}-\z_{s}\|^2+\|\z_{s}-\x_{s-1}\|^2 + 2\langle \x_{s}-\z_{s}, \z_{s}- \x_{s-1}\rangle\\
    \geq& (1-\alpha_{s}^{-1})\|\x_{s}-\z_{s}\|^2 + (1-\alpha_{s})\|\x_{s-1}-\z_{s}\|^2
\end{align*}
where the inequality follows from the Young's inequality with $0<\alpha_{s}<1$. Combining above inequalities, then we have
\begin{align*}
&\frac{(1-\alpha_s)}{2\gamma}\E_{s}\|\x_{s-1}-\z_s\|^2 \\
\leq& \E_s\bigg[ \Delta_s+\frac{(\alpha_{s}^{-1}-1)}{2\gamma}\|\x_{s}-\z_{s}\|^2 +\mathcal{E}_{s} \bigg]\\
\leq &\E_s[\Delta_s] + \frac{(\alpha_s^{-1}-1)}{2\gamma}\E_s[\|\x_s-\z_s\|^2] +\mathcal{E}_s\\
\leq& \E_{s}[\Delta_s] + \frac{(\alpha_s^{-1}-1)+\gamma\sigma}{\gamma\sigma}\mathcal{E}_s\\
\leq& \E_{s}[\Delta_s]  + \frac{(\alpha_s^{-1}-1)+\gamma\sigma}{\gamma\sigma}\left[4\theta^{-mK}(\phi(\x_{s-1}) - \phi(\x_*)) \right.\\
&\left. + 2\theta^{-mK}\hat L\|\x_{s-1} - \z_s\|^2 \right]
\end{align*}
where the first inequality follows from the definition $\Delta_s \coloneqq \phi(\x_{s-1}) - \phi(\x_s)$, and the third inequality uses the strong convexity of $f_s(\x)$, whose strong convexity parameter is $\sigma = \gamma^{-1} - \mu$.
Substituting $\alpha_s = 1/2$, $\gamma = 1/(2\mu)$, and $\sigma = \mu$, $\Lh\leq 2L$ and  $\theta^{-mK}\leq \mu/(24\hat L)$,  we have that
\begin{align*}
\frac{1}{8\gamma}\|\x_{s-1}-\z_s\|^2\leq \E_{s}[\Delta_s] + 12\theta^{-mK}(\phi(\x_{s-1}) - \phi(\x_*))
\end{align*} 
\end{proof}

\section{A Technical Lemma}\label{app:lem2}
\begin{lemma}\label{lem:2}
For a non-decreasing sequence $w_s, s=0, \ldots, S+1$, we have 
\[
\E\left[\sum_{s=1}^{S+1} w_s\Delta_s\right]\leq  \Delta_\phi w_{S+1}
\]
\end{lemma}
\begin{proof}
\begin{align*}
&\sum_{s=1}^{S+1} w_s\Delta_s = \sum_{s=1}^{S+1}w_s (\phi(\x_{s-1}) - \phi(\x_s)) \\
=& \sum_{s=1}^{S+1} (w_{s-1}\phi(\x_{s-1}) - w_s\phi(\x_s)) \\
&+ \sum_{s=1}^{S+1}(w_s - w_{s-1})\phi(\x_{s-1})\\
= &w_0 \phi(\x_0) - w_{S+1}\phi(\x_{S+1}) +\sum_{s=1}^{S+1}(w_s - w_{s-1})\phi(\x_{s-1})\\
=& \sum_{s=1}^{S+1}(w_s - w_{s-1})(\phi(\x_{s-1}) - \phi(\x_{S+1}))
\end{align*}
where the third equality follows from the extension that $w_0=0$. Taking expectation on both sides, we have
\begin{align*}
&\E\left[\sum_{s=1}^{S+1} w_s\Delta_s \right]\\
 =&\sum_{s=1}^{S+1}(w_s - w_{s-1})\E[(\phi(\x_{s-1}) - \phi(\x_{S+1}))]\\
\leq& \sum_{s=1}^{S+1}(w_s - w_{s-1})[\phi(\x_0) - \phi(\x_*)] \\
\leq& \Delta_\phi w_{S+1}
\end{align*}
where we use the fact that $\E[f_s(\x_s) -  f_s(\x_{s-1})]\leq 0$ (this is the objective value decreasing property of Katyusha) implying $\E[\phi(\x_s)  - \phi(\x_{s-1})]\leq 0$ and hence $\E[\phi(\x_{s})]\leq \phi(\x_0)$ for $s\geq 0$. 
\end{proof}

\section{Decomposition of LSP and TL1 }\label{app:dec}
It is easy to verify that for LSP, $r_1(\x) = \frac{\lambda}{\beta} \|\x\|_1$ and $r_2(x) = \lambda\sum_{i=1}^d(|x|/\beta- \log(\beta + |x|))$. For TL1, $r_1(\x) = \lambda \frac{\beta +1}{\beta}\|\x\|_1$ and $r_2(\x)= \lambda\sum_{i=1}^d\frac{(\beta+1)|x_i|^2}{\beta(\beta+|x_i|)}$. For smoothness of $r_2$ for both regularizers, we refer readers to~\cite{Wen2018}. 
\end{document}